\documentclass{amsart}
\usepackage{amsmath,amssymb}
\newtheorem{theorem}{Theorem}[section]
\newtheorem{lemma}[theorem]{Lemma}
\newtheorem{definition}[theorem]{Definition}
\newtheorem{proposition}[theorem]{Proposition}

\newtheorem{remark}[theorem]{Remark}
\newtheorem{step}{Step}

\newcommand\N{\mathbb N}
\newcommand\R{\mathbb R}
\newcommand\Z{\mathbb Z}
\newcommand{\Ha}{\mathcal H}
\newcommand{\cF}{\mathcal F}
\newcommand{\cB}{\mathcal B}
\newcommand{\cA}{\mathcal A}
\newcommand{\cM}{\mathcal M}
\newcommand{\bC}{{\bf C}}
\newcommand{\cN}{{\mathcal N}}
\newcommand{\E}{\mathbb E}

\newcommand{\rd}{\R^d}

\newcommand{\ep}{\varepsilon}

\newcommand{\loc}{\operatorname{loc}}

\newcommand{\sph}{S^{d-1}}

\newcommand{\conv}{\operatorname{conv}}

\newcommand{\BV}{\operatorname{BV}}
\newcommand{\FP}{\operatorname{FP}}
\newcommand{\Div}{\operatorname{div}}
\newcommand{\1}{{\bf 1}}
\newcommand{\per}{P}

\newcommand{\llc}{\;\halfsq\;}

\def\halfsq{\hbox{\kern1pt\vrule height 7pt\vrule width6pt height 0.4pt depth0pt\kern1pt}}
\def\ihalfsq{\hbox{\kern1pt \vrule width6pt height 0.4pt depth0pt
                   \vrule height 7pt \kern1pt}}
\begin{document}
\title{Random sets of finite perimeter}
\author{Jan Rataj}
\address{Charles University, Faculty of Mathematics and Physics, Sokolovsk\'a 83, 18675 Praha 8, Czech Republic}
\email{rataj@karlin.mff.cuni.cz}
\thanks{Supported by grant GA\v CR 201/10/0472}

\begin{abstract}
An approach to modelling random sets with locally finite perimeter as random elements in the corresponding subspace of $L^1$ functions is suggested. A Crofton formula for flat sections of the perimeter is shown. Finally, random processes of particles with finite perimeter are introduced and it is shown that their union sets are random sets with locally finite perimeter. 
\end{abstract}

\keywords{}
\subjclass[2000]{60D05; 60B99}

\maketitle

\section{Introduction}
In stochastic geometry, random sets in a Euclidean space are standardly considered as random closed sets, which is a concept introduced by Matheron \cite{Ma75}. For later surveys, see \cite{SKM95}, \cite{Mo05}. A random closed set is a random element from the family $\cF$ of all closed subsets of the Euclidean space $\rd$ equipped with the Fell topology, whose basis is generated by the sets $\cF_G:=\{F\in\cF:\, F\cap G\neq\emptyset\}$, $G\subset\rd$ open, and $\cF^K:=\{F\in\cF:\, F\cap K=\emptyset\}$, $K\subset\rd$ compact. For statistical inference, various test sets are used to observe whether the random closed set hits or misses them. The model of a random closed set is used for rather different phenomena as point patterns, unions of segments, lines, curves, surfaces, or ``full-dimensional'' sets with boundaries satisfying some regularity conditions. Of course, the type of the test set must reflect the nature of the random closed set which is analyzed.

Full-dimensional random objects with regular boundary are usually modelled as sets from the extended convex ring (or, more generally, unions of sets with positive reach). Nevertheless, if only first order geometric quantities of the boundary as surface area are measured, the model assumption is too restrictive. It has already been observed that the framework of sets of finite perimeter due to Caccioppoli is probably the most natural and general one for such a purpose. In particular, Ambrosio et al. \cite{ACV08} considered the outer Minkowski content of such (random) sets, or Galerne \cite{Ga11} extended some properties of the covariogram to this setting. As far as we know, however, up to now, the Matheron's concept of a random closed set was considered, with certain additional assumptions.

The aim of this note is to suggest another concept of a random set, namely a random set of finite perimeter, which should serve as a sufficiently general and suitable model whenever quantities derived from the surface area are considered. The first substantial difference from the Matheron's approach is that sets of finite perimeter (represented by their indicator function) are considered as elements of the Lebesgue space $L^1$. Therefore, we do not distinguish two sets whose symmetric difference has Lebesgue measure zero. This is not unreasonable if we admit that random sets are usually observed in a lattice pixel approximation on the screen. Also, viewing random sets as their indicator functions belonging to a larger space of functions of bounded variation makes it possible to use approximations with continuous or even smooth functions and apply new techniques not available for point sets. Of course, when dealing with point patterns or lower-dimensional objects, different models should be applied.

In order to deal with random sets of (locally) finite perimeter we have to equip the family of sets of (locally) finite perimeter with a topology determining measurability and convergence. We suggest to consider the strict convergence which assures  both convergence in $L^1$ and convergence of perimeter. Roughly speaking, the convergence in this topology of sets guaranties the convergence of both volume and surface area, in contrast to the Fell topology.

As one particular example of classical properties which can be transformed to our setting, we prove the Crofton formula for perimeter in Section~3 and apply it to random sets in Theorem~\ref{T_Croft}. 

In the last two sections we introduce the two basic stochastic models, random sets with (locally) finite perimeter and random $\FP$-processes (as processes of particles with finite perimeter). We show some basic properties of convergence in distribution of random sets with locally finite perimeter. Our last result is that the union set of a random $\FP$-process is a random set with locally finite perimeter.

\section{Sets of finite perimeter}
Our basic setting is the $d$-dimensional Euclidean space $\rd$ with norm $|\cdot|$. Given $0\leq k\leq d$, $\Ha^k$ denotes the $k$-dimensional Hausdorff measure in $\rd$. In particular, $\Ha^d$ agrees with the $d$-dimensional Lebesgue measure.

In this section, we summarize the necessary definitions and properties of sets with finite perimeter. We refer to \cite{AFP00} or \cite{Zi89}.

The {\it distributional derivative} of a function $f\in L^1\equiv L^1(\rd)$ is the functional 
$$Df=(D_1f,\ldots,D_df):C^1_c(\rd)\to\rd$$
defined by
$$D_if (\phi):=-\int f\cdot\frac{\partial \phi}{\partial x_i}\, d\Ha^d,\quad i=1,\ldots, d$$
($C^1_c(\rd)$ denotes the space of $C^1$-smooth functions on $\rd$ with compact support).
If $Df$ can be represented as a (locally) finite Radon vector-valued measure, we say that $f$ has ({\it locally}) {\it bounded variation}. Its variation measure
$$|Df|(\cdot)=\sup\{ Df(E_1)+\cdots +Df(E_n):\, E=E_1\cup\cdots\cup E_n\text{ measurable finite partition}\}$$
is then a nonnegative (locally) finite Radon measure, and its total variation
$$Vf:=|Df|(\rd)=\sup\left\{ \int f\cdot \Div \varphi\,d\Ha^d:\, \varphi\in C^1_c(\rd,\rd), \sup_x|\varphi(x)|\leq 1\right\}.$$
is called the {\it variation} of $f$. (Here $\Div\varphi (x)=\sum_{i=1}^d\frac{\partial\varphi^i}{\partial x_i}(x)$ is the divergence of a mapping $\varphi=(\varphi^1,\ldots,\varphi^d):\rd\to\rd$.)

The vector-valued measure $Df$ is clearly absolutely continuous with respect to its total variation $|Df|$ and if $\Delta_f\in L^1(|Df|)$ is its Radon-Nikodym density, it satisfies $\Delta_f(x)\in\sph$ for $|Df|$-almost all $x$.

A Lebesgue measurable set $A\subset\rd$ is said to have ({\it locally}) {\it finite perimeter} if its indicator function $\1_A$ has (locally) finite variation. The {\it perimeter} of $A$ is defined as
$$\per (A):=V\1_A=|D\1_A|(\rd).$$
Let $A\subset\rd$ have locally finite perimeter. The {\it reduced boundary} of $A$, $\partial^-A$, consists of all points $x\in\rd$ such that the limit 
$$\nu_A(x):=-\lim_{r\to 0}\frac{D\1_A(B(x,r))}{|D\1_A|(B(x,r))}$$
exists and satisfies $|\nu_A(x)|=1$. (We assume implicitly here that $|D\1_A|(B(x,r))>0$ for all $r>0$.)
$\nu_A(x)$ is called the {\it generalized exterior normal} to $A$ at $x$. In fact, $\nu_A$ is a version of the density $\Delta_{\1_A}$ of $D\1_A$ with respect to $|D\1_A|$, see \cite[\S5.5]{Zi89}.
If $A$ has locally finite perimeter then $\partial^-A$ is countably $(d-1)$-rectifiable and 
$$|D\1_A|(B)=\Ha^{d-1}(B\cap \partial^-A),\quad B\in\cB^d;$$
in particular,
$$\per (A)=\Ha^{d-1}(\partial^-A).$$

The {\it measure-theoretical boundary} $\partial^MA$ of a set $A\subset\rd$ is defined as the set of all points $x\in\rd$ at which the Lebesgue density of $A$ in neither $0$ nor $1$. If $A$ has finite perimeter then $\partial^-A\subset\partial^MA$, $\Ha^{d-1}(\partial^MA\setminus\partial^-A)=0$ and, hence,
$$\per (A)=\Ha^{d-1}(\partial^M A).$$

Let us call the set
$$N(A):=\{ (x,\nu_A(x)):\, x\in\partial^-A\}\subset \rd\times\sph$$
{\it unit normal bundle} of $A$. (Notice that $N(A)$ is neither closed, nor countably $(d-1)$-rectifiable, in general, in contrast to the unit normal bundle defined in the classical setting for convex bodies or sets with positive reach.) 
We consider the measure
$$\bC_{d-1}(A,\cdot):=(\Ha^{d-1}\llc \partial^-A) \psi^{-1},$$
where the Borel measurable mapping $\psi: x\mapsto (x,\nu_A(x))$ is defined on $\partial^-A$. $\bC_{d-1}(A,\cdot)$ is a locally finite Borel measure on $\rd\times\sph$ and it is finite iff $A$ has finite perimeter.
If $A$ has finite perimeter we define the {\it normal measure} of $A$
$$S_{d-1}(A,\cdot):= \bC_{d-1}(A,\rd\times\cdot);$$
it is a finite Borel measure on $\sph$. Clearly,
$$\bC_{d-1}(A,\rd\times\sph)=S_{d-1}(A,\sph)=\per(A).$$

\section{Flat sections of FP-sets and a Crofton formula}
Let $L\in G(d,j)$ be a fixed $j$-subspace, and let $p_L$ be the orthogonal projection from $\rd$ onto $L$. Given a function $f\in L^1$, we denote by
$$D_Lf(\cdot):=p_L(Df(\cdot))$$
the {\it distributional derivative of $f$ with respect to the subspace} $L$. (Note that $D_Lf$ takes into account only directional derivatives in directions from $L$.)

Immediately from the definition we see that if $f$ has bounded variation then $D_Lf$ is a finite Radon measure and its total variation is related to that of $Df$ by
\begin{equation}   \label{E11}
|D_Lf|(B)=\int_B\| p_L\Delta_f(x)\|\, |Df|(x),
\end{equation}
where $\Delta_f$ is the Radon-Nikodym derivative of the vector measure $Df$ with respect to its total variation $|Df|$.

If $E$ be a $j$-flat in $\rd$ and $f\in L^1(E)$, then we denote by $D^{(E)}f$ the distributional derivative of $f$ in $E$. This vector measure depends on the chosen orientation of $E$, whereas its total variation $|D^{(E)}f|$ does not.

\begin{lemma}  \label{L11}
If $L\in G(d,j)$, $f$ has bounded variation and $B$ is a Borel subset of $\rd$ then
$$|D_Lf|(B)=\int_{L^\perp}|D^{(L+z)}(f|L+z)|(B\cap(L+z))\, \Ha^{d-j}(dz).$$
\end{lemma} 

\begin{proof} Note that if $\phi\in C^1$ then the restriction $\phi|L+z\in C^1(L+z)$ and the gradient satisfies
$$\nabla(\phi|L+z)(x)=p_L\nabla\phi(x),\quad x\in L+z.$$
Thus,
\begin{eqnarray*}
D^{(L+z)}(f|L+z)(\phi|L+z)&=&-\int_{L+z} f\cdot\nabla(\phi|L+z)\, dx\\
&=&-\int_{L+z} f\cdot p_L\nabla\phi\, dx.
\end{eqnarray*}
Hence, integrating over $z\in L^\perp$, we get from the Fubini theorem
$$\int_{L^\perp}D^{(L+z)}(f|L+z)(\phi|L+z)\, \Ha^{d-j}(dz)=\int_{\rd} f\cdot p_L\nabla\phi\, dx=D_Lf(\phi).$$
The assertion follows.
\end{proof}

\begin{theorem}[Crofton formula for perimeter]  \label{T_Crof}
If $f$ has bounded variation and $B\subset\rd$ Borel then
$$\int_{\cA(d,j)}|D^{(E)}(f|E)|(B\cap E)\, \mu_j^d(dE)=c_{d,j}|Df|(B),$$
where
$$c_{d,j}=\frac{\Gamma(\tfrac{2d-j}2)\Gamma(\tfrac{j+1}2)}{\Gamma(\tfrac{d+1}2)\Gamma(\tfrac{d}{2})}.$$
In particular, if $A\in\FP$ then
$$\int_{\cA(d,j)}\per(A\cap E)\, \mu_j^d(dE)=c_{d,j}\per(A).$$
\end{theorem}

\begin{proof}
Integrating the equality from Lemma~\ref{L11} over $L\in G(d,j)$, we get using \eqref{E11}
$$\int_{\cA(d,j)}|D^{(E)}(f|E)|(B\cap E)\, \mu_j^d(dE)=\int_B\int_{G(d,j)}\| p_L(\Delta_f(x))\|\, \nu_j^d(dL)\, dx.$$
The inner integral does not depend on the vector $\Delta_f(x)$; a routine calculation verifies that
$$\int_{G(d,j)}\| p_L(\Delta_f(x))\|\, \nu_j^d(dL)=c_{d,j}^{-1},$$
and the assertion follows.
\end{proof}

\section{Spaces $\BV$ and $\FP$}
Let $\BV$ denote the set of all functions $f\in L^1$ with bounded variation, and $\BV_{\loc}$ the set of all functions $f\in L^1_{\loc}$ with locally bounded variation.
Further, we denote by $\FP\subset\BV$ the subfamily of sets with finite perimeter, and by $\FP_{\loc}\subset\BV_{\loc}$ the subfamily of sets with locally finite perimeter.

We equip $\BV$ with the {it topology of strict convergence} defined as follows. If $f_i,f\in\BV$ then $f_i\stackrel{s}{\to}f$ iff $f_i\to f$ in $L^1(\Omega)$ and $Vf_i\to Vf$.
The strict convergence implies weak convergence $Df\stackrel{w}{\to} Df$, but not vice versa (see \cite[\S3.1]{AFP00}). Strict convergence is induced by the metric
\begin{equation}  \label{sc}
d_s(f,g)=\int |f-g|\, dx+|Vf-Vg|,\quad f,g\in\BV .
\end{equation}
Strict convergence implies weak convergence not only of the distributional derivatives, but even of their variation measures, i.e.,
\begin{equation}  \label{sw}
f_i\stackrel{s}{\to}f\implies |Df_i|\stackrel{w}{\to}|Df|,
\end{equation}
see \cite[Proposition~3.15]{AFP00}.

Analogously, we equip $\BV_{\loc}$ with the {\it locally strict convergence} given by $f_i\stackrel{ls}{\to}f$ iff $f_i\to f$ in $L^1_{\loc}(\Omega)$ and the measures $|Df_i|$ converge to $|Df|$ vaguely (we write $|Df_i|\stackrel{v}{\to}|Df|$). In the sequel, whenever speaking about the space $\BV$ ($\BV_{\loc}$) or its subspace $\FP$ ($\FP_{\loc}$), we will mean the topology induced by strict (locally strict) convergence.

\begin{remark} \rm
The metric $d_s$ on $\BV$ is not complete, as can be seen from the following example: the functions $f_i(x)=\sin(ix)/i$, $x\in (0,2\pi)$, and $f_i(x)=0$ otherwise, converge to $0$ in $L^1$ and their variations are constant and nonzero: $Vf_i=4$ for all $i\in\N$. Thus, the sequence $(f_i)$ is Cauchy in $d_s$ and its limit cannot be nothing else than the zero function, but the variations would not converge. On the other hand, the space $\BV$ with the Borel $\sigma$-algebra $\cB(\BV)$ of the metric $d_s$ is a {\it standard Borel space} in the sense that there exists another metric $\rho$ on $\BV$ having the same Borel sets as $d_s$. This can be seen as follows: Since the variation $f\mapsto Vf$ is lower semicontinuous in $L^1$ (\cite[Remark~3.5]{AFP00}), the unit $d_s$-ball is a $F_\sigma$-set in $L^1$ and, hence, the Borel sets in $(\BV,d_s)$ agree with the Borel sets of $L^1$ intersected with $\BV$. Thus, a result of descriptive set theory \cite[Corollary~13.4]{Ke95} implies that $(\BV,\cB(\BV))$ is a standard Borel space.
\end{remark}

Let $\cM$ be the space of locally bounded Borel measures on $\rd$ with topology of vague convergence, and let $\cM_b$ be its subspace of bounded Borel measures on $\rd$ (its induced topology coincides with that of weak convergence).

\begin{proposition}  \label{C_cont}
The assignment
$$A\mapsto \bC_{d-1}(A,\cdot)$$
defines a continuous mapping from $\FP$ to $\cM_b$ and from $\FP_{\loc}$ to $\cM$.
\end{proposition}

\begin{proof}
Assume that $A_i,A\in\FP$ and that $A_i\stackrel{s}{\to} A$, $i\to\infty$. Then, using \cite[Proposition~3.15]{AFP00}, we get that $\bC_{d-1}(A_i,\cdot)\stackrel{w}{\to}\bC_{d-1}(A,\cdot)$, $i\to\infty$, proving the first statement. The second statement follows analogously.
\end{proof}

An important fact for applications in stochastic geometry is that the spaces $\FP$ and $\FP_{\loc}$ are closed with respect to finite unions and intersections (unless, e.g., sets with piecewise smooth boundaries).

\begin{proposition}  \label{P-union}
If $A,B\in\FP_{\loc}$ then both $A\cup B$, $A\cap B\in\FP_{\loc}$ and
$$|D\1_{A\cup B}|(\cdot)+|D\1_{A\cap B}|(\cdot)\leq |D\1_{A}|(\cdot)+|D\1_{B}|(\cdot).$$
Consequently, if $A,B\in\FP$ then $A\cup B,A\cap B\in\FP$ and
\begin{equation}  \label{P1-1}
\per(A\cup B)+\per (A\cap B)\leq \per (A)+\per (B).
\end{equation}
Further, if $(A_i)$ is a finite or countable family of sets of finite perimeter then
\begin{equation}  \label{P1-2}
|D\1_{\bigcup_iA_i}|(\cdot)\leq\sum_i|D\1_{A_i}|(\cdot).
\end{equation}
\end{proposition} 

\begin{proof}
Inequality \eqref{P1-1} was shown in \cite[Proposition~1]{ACMM99} by the following argument. From the definition of the measure-theoretic boundary we get the inclusions
\begin{eqnarray*}
\partial^M(A\cup B)\cup\partial^M(A\cap B)&\subset& \partial^MA\cup\partial^MB,\\
\partial^M(A\cup B)\cap\partial^M(A\cap B)&\subset& \partial^MA\cap\partial^MB.
\end{eqnarray*}
Since for any $E\subset\rd$ measurable, $|D\1_E|(\cdot)=\Ha^{d-1}(\partial^ME\cap\cdot)$, the first inequality follows, and \eqref{P1-1} is a consequence.
The second inequality, \eqref{P1-2}, can also be found in \cite{ACMM99}, it follows from \eqref{P1-1} and from the lower semicontinuity of the perimeter with respect to the $L^1_{\loc}$ convergence.
\end{proof}

The set operations of union and intersection are of principal importance in stochastic geometry. In the Fell topology, the mapping $(A,B)\mapsto A\cup B$ is continuous \cite[Corollary~1 of Theorem~1.2.2]{Ma75} and $(A,B)\mapsto A\cap B$ is upper semicontinuous \cite[Corollary~1 of Proposition~1.2.4]{Ma75}.

\begin{proposition}  \label{meas}
The mappings
\begin{eqnarray*}
\cap &:&(A,B)\mapsto A\cap B,\\
\cup &:&(A,B)\mapsto A\cup B
\end{eqnarray*}
are measurable from $\FP\times\FP$ to $\FP$ and from $\FP_{\loc}\times\FP_{\loc}$ to $\FP_{\loc}$.
\end{proposition}

\begin{proof}
First, note that the mappings $\cup$ and $\cap$ are continuous from $L^1\times L^1$ to $L^1$ (and from $L^1_{\loc}\times L^1_{\loc}$ to $L^1_{\loc}$). This follows from the inclusions
\begin{eqnarray*}
(A\cup B)\Delta(A'\cup B')&\subset&(A\Delta A')\cup(B\Delta B'),\\
(A\cap B)\Delta(A'\cap B')&\subset&(A\Delta A')\cup(B\Delta B').
\end{eqnarray*}
Further, we observe that the mappings
$$(A,B)\mapsto \per(A\cup B),\quad (A,B)\mapsto\per(A\cap B)$$
are lower semicontinuous on $\FP\times \FP$. This follows from the first observation and from the lower semicontinuity of variation with respect to the strict topology (\cite[Remark~3.5]{AFP00}).
Hence, the mappings
$$(A,B,A',B')\mapsto d_s(A\cup B,A'\cup B'),\quad d_s(A\cap B,A'\cap B')$$
are measurable on $\FP^4$, where $d_s$ is the metric \eqref{sc} inducing strict convergence. Hence, the measurability of $\cup$ and $\cap$ on $\FP\times\FP$ follows. The case of $\FP_{\loc}\times\FP_{\loc}$ needs some further standard consideration which will be left to the reader.
\end{proof}

\section{Random sets of finite perimeter and $\FP$-processes}

Let $(\Omega,\Sigma,\Pr)$ be a standard probability space. A {\it random set with} ({\it locally}) {\it finite perimeter} is a measurable mapping
$$X:(\Omega,\Sigma,\Pr)\to (\FP,\cB(\FP))\, (\FP_{\loc},\cB(\FP_{\loc})), \text{ respectively}.$$
The probability measure $\Pr X^{-1}$ on $(\FP,\cB(\FP))\, (\FP_{\loc},\cB(\FP_{\loc}))$, \text{ respectively}, is called the {\it distribution} of $X$.

A random set $X$ with locally finite perimeter is said to be {\it stationary} if its distribution $\Pr X^{-1}$ is translation invariant (i.e., $\Pr X^{-1}=\Pr (X+z)^{-1}$ for all $z\in\rd$). In such a case, the measure $\E |D\1_X|$ is translation invariant, and it is, hence, a multiple of the Lebesgue measure whenever it is locally finite. We can thus define the {\it specific perimeter} $\bar{P}$ of $X$ through
$$\E |D\1_X|(B) =\bar{P}(X) \Ha^d(B),$$
where $B\in\cB^d$ is any set of finite positive Lebesgue measure. (Note that $\bar{P}(X)$ may take the value $\infty$.)

It is clear that, $X$ being a random set with locally finite perimeter, $\bC_{d-1}(X,\cdot)$ is a random measure and if $X$ is stationary then we obtain by standard methods that $\E\bC_{d-1}(X,\cdot)$ factorizes in a product of the Lebesgue measure with a measure on $\sph$. Assuming that $\bar{P}(X)<\infty$, we can thus define the {\it specific area measure} $\bar{S}_{d-1}(X,\cdot)$ through
$$\E\bC_{d-1}(X,B\times\cdot)=\Ha^d(B) \bar{S}_{d-1}(X,\cdot),$$
where, again, $B\in\cB^d$ is any set of finite positive Lebesgue measure. $\bar{S}_{d-1}(X,\cdot)$ is a finite Borel measure of the unit sphere and its total measure is $\bar{S}_{d-1}(X,\sph)=\bar{P}(X)$.

In the following, $\stackrel{d}{\to}$ denotes convergence of random variables in distribution. Further, if $X$ is a random set with locally finite perimeter, we call a Borel set $B\subset\rd$ {\it $X$-continuous} if $\Ha^d(\partial B)=0$ and $|D\1_X|(\partial B)=0$ almost surely.
 
\begin{proposition}
Let $X_i,X$ be random sets with locally finite perimeter and assume that $X_i\stackrel{d}{\to}X$. Then we have:
\begin{enumerate}
\item[{\rm (i)}] $\Ha^d(X_i\cap K)\stackrel{d}{\to}\Ha^d(X\cap A)$ for any $K\subset\rd$ compact;
\item[{\rm (ii)}] $\int g(x)\, |D\1_{X_i}|(dx)\stackrel{d}{\to}\int g(x)\, |D\1_{X}|(dx)$ for any $g\in C^1_c(\rd)$;
\item[{\rm (iii)}] $|D\1_{X_i}|(K)\stackrel{d}{\to}|D\1_{X}|(K)$ for any $X$-continuous compact set $K\subset\rd$.
\end{enumerate}
Further, $\bC_{d-1}(X_i,\cdot)$ and $\bC_{d-1}(X,\cdot)$ are random locally bounded measures and we have
\begin{enumerate}
\item[{\rm (iv)}] $\bC_{d-1}(X_i,\cdot)\stackrel{d}{\to}\bC_{d-1}(X,\cdot)$, $S_{d-1}(A,\cdot)\stackrel{d}{\to}S_{d-1}(A,\cdot)$.
\end{enumerate}
\end{proposition}

\begin{proof}
The assumption $X_i\stackrel{d}{\to}X$ implies by definition that $H(X_i)\stackrel{d}{\to} H(X)$ for any continuous function $H$ on $\FP_{\loc}$. Since the function $A\mapsto\Ha^d(A\cap K)$ is continuous on $\FP_{\loc}$ for any compact $K\subset\rd$, we obtain (i). To verify (ii), we use the function $A\mapsto \int g(x)\, |D\1_A|(dx)$, which is again continuous on $\FP_{\loc}$ for any $g\in C^1_c(\rd)$. In order to show (iii), we observe that $A\mapsto |D\1_A|(A\cap K)$ is continuous on the set $\{A\in\FP_{\loc}:\, |D\1_A|(\partial K)=0\}$ whenever $K$ is a compact set with $\Ha^d(\partial K)=0$. Finally, (iv) follows directly from Proposition~\ref{C_cont}.
\end{proof}

The Crofton formula for perimeter gives us the following classical stereological relation between specific perimeter of a random set and of its flat sections.

\begin{theorem} \label{T_Croft}
Let $X$ be a stationary random set of locally finite perimeter and with finite specific perimeter. Then, for any $1\leq j\leq d$ and for any $B\in\cB^d$,
$$\int_{\cA(d,j)}\E |D\1_{X\cap E}|(B\cap E)\, \mu_j^d(dE)=c_{d,j}\E|D\1_X|(B).$$
In particular,
$$\int_{\cA(d,j)}\bar{P}^{(E)}(X\cap E)\, \mu_j^d(dE)=c_{d,j}\bar{P}(X).$$
\end{theorem}

\begin{proof}
Follows from Theorem~\ref{T_Crof} applying the expectation on both sides and Tonelli's theorem for exchanging integral with expectation.
\end{proof}

\section{Random $\FP$-processes}

A  random $\FP$-process will be a point process on $\FP$ in the sense of Mecke \cite{Me79}, see also Ripley \cite{Ri76}. In particular, we consider the triple
$$(\FP,\cB(\FP),\cB_0),$$
called {\it bounded space} in \cite{Me79}, where $\cB_0$ is the subfamily of $\cB(\FP)$ consisting of those $U\in\cB(\FP)$ for which there exist a compact set $K\subset\rd$ such that $|D\1_A|(K)>0$ for all $A\in U$. Note that the measurable space $(\FP,\cB(\FP))$ is {\it full} in the sense on \cite{Me79} since each standard Borel space is full (cf.\ \cite[Theorem~1]{Me79}), and $\cB_0$ defined above clearly satisfies the requirements from \cite{Me79} since the sets 
$$U_n=\{ A\in\FP:\, |D\1_A|([-n,n]^d)>0\}\in\cB_0$$ 
cover the whole space $\FP$ and $\cB_0=\bigcup_n\{ U:\, U\in\cB(\FP),\, U\subset U_n\}$.

\begin{definition} \rm
A  {\it random $\FP$-process} is a measurable mapping
$$\Phi: (\Omega,\Sigma,\Pr)\to (\cN^{\#}_{\FP},{\frak N}^{\#}_{\FP}),$$
where $\cN^{\#}_{\FP}$ is the set of all integer-valued (nonnegative) Borel measures on $\FP$ that are finite on $\cB_0$, and ${\frak N}^{\#}_{\FP}$ is the smallest $\sigma$-algebra on $\cN^{\#}_{\FP}$ such that all mappings $\nu\mapsto\nu(U)$ are measurable, $U\in\cB_0$.
\end{definition}

The approach of Ripley and Mecke does not use any particular topology (metric) on $\FP$. Nevertheless, by \cite[Theorem~2]{Ri76}, there exists a complete metric on $\FP$ making it a locally compact space with the same Borel $\sigma$-field, $\cB(\FP)$, and such that $\cB_0$ agrees with both relatively compact sets as well as metrically bounded sets. Thus, the theory of \cite{DVJ08} can be applied. In particular, we can define the {\it weak$^{\#}$} (weak-hash) convergence on $\cN^{\#}_{\FP}$ by $\mu_i\to\mu$ weakly$^{\#}$ if $\int g\, d\mu_i\to\int g\, d\mu$ for any bounded continuous function $g$ on $\FP$ with support in $\cB_0$, and we get that ${\frak N}^{\#}_{\FP}$ agrees with the Borel $\sigma$-field of the weak$^{\#}$ topology (cf.\ \cite[Proposition~9.1.IV]{DVJ08}).

A point process $\Phi$ on $\BV$ is said to be {\it stationary} if its distribution is invariant with respect to the shift operation in $\rd$. A stationary point process $\Phi$ has an intensity $\gamma>0$ and {\it typical grain} $Z_0\in \FP$ (a random set with finite perimeter) and its mean characteristics are denoted as
\vspace{2mm}

\begin{tabular}{ll}
$\bar{V}_d(\Phi):=\E\Ha^d(Z_0)$ &mean volume,\\
$\bar{P}(\Phi):=\E P(Z_0)$ &mean perimeter,\\
$\bar{S}_{d-1}(\Phi,\cdot):=\E S_{d-1}(Z_0,\cdot)$&mean area measure of $\Phi$.
\end{tabular}

\begin{theorem}
Let $\Phi$ be a $\FP$-process.
Then, the union set
$$X:=\bigcup_{A\in\Phi}A$$
is a random set with locally finite perimeter. If $\Phi$ is stationary then, so is $X$.
\end{theorem}

\begin{proof}
We have to show that the mapping
$$\bigcup: \phi\mapsto \bigcup\phi = \bigcup_{A\in\phi}A$$
is measurable from $\cN^{\#}_{\FP}$ to $\FP_{\loc}$. First, note that whenever $\phi\in\cN^{\#}_{\FP}$ then $\bigcup\phi$ is a Lebesgue measurable subset of $\rd$, hence, its indicator function belongs to $L^1_{\loc}$ and its perimeter is locally bounded, since by \eqref{P1-2} we have
$$|D\1_{\bigcup\phi}|(K)\leq\int |D\1_A|(K)\, \Phi(dA)<\infty,\quad K\subset\rd\text{ compact}$$
(note that the last integral is in fact only a finite sum since $\Phi$ is finite on $\cB_0$).

It remains to verify the measurability of $\bigcup$. This will be done in two steps.

\begin{step}
The mapping $\bigcup$ is continuous from $\cN^{\#}_{\FP}$ to $L^1_{\loc}$.
\end{step}
To see this, let $\phi_i,\phi\in\cN^{\#}_{\FP}$ be such that $\phi_i\to\phi$ weakly$^{\#}$. We shall show that
\begin{equation} \label{eee}
\Ha^d(\bigcup\phi_i\Delta\bigcup\phi_i)\to 0\text{ for any }K\subset\rd\text{ compact}.
\end{equation}
Fix a compact set $K\subset\rd$ and consider the function
$$h:A\mapsto\Ha^d(K\cap A\setminus\bigcup\phi),\quad A\in\FP.$$
The function $h$ is clearly continuous, bounded and has support in $\cB_0$, hence, $\int h\, d\phi_i\to\int h\, d\phi=0$. If follows that $\Ha^d(\bigcup\phi_i\setminus\bigcup\phi_i)\to 0$. We shall finish the proof of \eqref{eee} by contradiction. Assume that $\limsup_i\Ha^d(\bigcup\phi\setminus\bigcup\phi_i)>0$. Then, there exists a measurable set $B\subset K\cap\bigcup\phi$ and a subsequence $(i_k)$ such that $\Ha^d(B\cap \bigcup\phi_{i_k})=0$ for all $k$. But, since the function $h_B:A\mapsto\Ha^d(B\cap A)$ is continuous, bounded and with support in $\cB_0$, we have $\int h_B\, d\phi_{i_k}\to\int h_B\, d\phi\geq\Ha^d(B)$, a contradiction.

\begin{step}
For any $g\in C_c(\rd)$, the function $\phi\mapsto\int g\, |D\1_{\bigcup\phi}|$ is lower semicontinuous.
\end{step}
It is known that $f\mapsto\int g\, |Df|$ is lower semicontinuous on $L^1_{\loc}$, cf.\ \cite[Remark~3.5]{AFP00}. Composing this mapping with the smooth mapping $\phi\mapsto\bigcup\phi$ from Step~1, we obtain Step~2.

Taking into account the definition of locally strict convergence, it is clear that Steps 1 and 2 imply already the measurability of $\phi\mapsto\bigcup\phi$. The statement about stationarity is obvious and the proof is thus finished.
\end{proof}

In general, it is not possible to relate the specific perimeter (area measure) of the union set to the mean perimeter (area measure) of $\Phi$. This can be done in the case of a Poisson process.

\section{examples}
\subsection{Random approximations of convex bodies}
We shall work now in dimension two, though the same procedure can be applied in general dimension.

Let $K\subset\R^2$ be a fixed convex body with nonempty interior. Let, further, $L=\rho(\zeta+\Z^2)$ be a randomly shifted and rotated integer lattice (here $\rho$ is a uniform random rotation and $\zeta$ a uniform random point from $[0,1]^2$). We consider a rescaled lattice $tL$ with $t>0$ small and use it for approximating $K$.
\subsubsection*{Pixel approximation} The set
$$Z_1^t=\bigcup_{z\in tL\cap K} (z+t\rho[-\tfrac 12,\tfrac 12]^2)$$
is the union of pixels whose centres lie in $K$. It is a random closed set in the sense of Matheron as well as a random set with finite perimeter. $Z_1^t$ converges to $K$ (almost surely, as well as in distribution) in the Fell topology as $t\to 0$, but not in the space $\FP$ since, of course, the perimeter of $Z_1^t$ does not converge to that of $K$.
\subsubsection*{Convex hull of pixel centres}
Consider the set
$$Z_2^t=\conv(tL\cap K)$$
(the convex hull of lattice points lying in $K$). $Z_2^t$ is again a random closed set as well as random set with finite perimeter, and it is not difficult to show that its perimeter converges to that of $K$. Therefore, $Z_2^t\stackrel{d}{\to}K$ as $t\to 0$, both in the Fell topology and in the strict topology of $\FP$.

\subsection{Swiss cheese}
Let $(\xi_i)$ be a sequence of i.i.d.\ uniform random points from $[0,1]^d$, and $0<\ep<\frac 12$. Consider the set
$$Z^{\ep}=\bigcup_{i=1}^{\infty}U(\xi_i,\ep/2^i),$$
where $U(x,r)$ denotes the open Euclidean ball of centre $x$ and radius $r$. Applying Proposition~\ref{P-union}, we get that $P(Z^{\ep})\leq\pi\ep^2$, hence, $Z^{\ep}$ is a random set with finite perimeter. On the other hand, the closure of $Z^{\ep}$ covers the whole cube $[0,1]^d$ almost surely since the i.i.d.\ sequence $(\xi_i)$ s dense in $[0,1]^d$ almost surely. Therefore, there seems to be no way how to consider $Z^{\ep}$ as a random closed set.

The set
$$\Xi^{\ep}=[0,1]^d\setminus Z^{\ep}$$
is again a random set of finite perimeter (this should resemble the ``Swiss cheese'' if $d=3$, as a block of cheese with infinitely many small circular holes). Note that $\Xi^{\ep}$ is closed and can be considered as a random closed set in the sense of Matheron, as well. We have $\Xi^{\ep}\to[0,1]^d$ as $\ep\to 0$ both in the Fell topology as well as in the strict topology on $\FP$. Note that, however, the topologies of $\Xi^{\ep}$ and $[0,1]^d$ are completely different.

We know that the specific perimeter $\bar{P}(\Xi^{\ep})$ is finite, but it seems to be difficult to obtain the exact value.

\end{document}